\documentclass[conference,10pt]{IEEEtran}
\IEEEoverridecommandlockouts
\usepackage{cite}
\usepackage{amsmath,amssymb,amsfonts, amsthm}
\usepackage{algorithmic}
\usepackage{graphicx}
\usepackage{textcomp}
\usepackage{xcolor}
\usepackage{hyperref}

\def\BibTeX{{\rm B\kern-.05em{\sc i\kern-.025em b}\kern-.08em
    T\kern-.1667em\lower.7ex\hbox{E}\kern-.125emX}}

\newtheorem{theorem}{Theorem}
\newtheorem{proposition}[theorem]{Proposition}
\newtheorem{lemma}[theorem]{Lemma}
\newtheorem{corollary}[theorem]{Corollary}
\newtheorem{definition}[theorem]{Definition}
\newtheorem{remark}[theorem]{Remark}
\newtheorem{example}[theorem]{Example}

\begin{document}

\title{On the affine permutation group of certain decreasing Cartesian codes
\thanks{Hiram H. L\'opez was partially supported by the NSF grants DMS-2201094 and DMS-2401558.}
}

\author{\IEEEauthorblockN{Eduardo Camps-Moreno}
\IEEEauthorblockA{\textit{Virginia Tech} \\
Blacksburg, VA, USA \\
e.camps@vt.edu}
\and
\IEEEauthorblockN{Hiram H. L\'opez}
\IEEEauthorblockA{\textit{Virginia Tech} \\
Blacksburg, VA, USA \\
hhlopez@vt.edu}
\and
\IEEEauthorblockN{Eliseo Sarmiento}
\IEEEauthorblockA{\textit{Instituto Polit\'ecnico Nacional} \\
Mexico City, Mexico \\
esarmiento@ipn.mx}
\and
\IEEEauthorblockN{Ivan Soprunov}
\IEEEauthorblockA{\textit{Cleveland State University} \\
Cleveland, OH, USA \\
i.soprunov@csuohio.edu}}

\maketitle

\begin{abstract}
A decreasing Cartesian code is defined by evaluating a monomial set closed under divisibility on a Cartesian set. Some well-known examples are the Reed-Solomon, Reed-Muller, and (some) toric codes. The affine permutations consist of the permutations of the code that depend on an affine transformation. In this work, we study the affine permutations of some decreasing Cartesian codes, including the case when the Cartesian set has copies of multiplicative or additive subgroups.
\end{abstract}


\section{Introduction}
Let $\mathbb{F}_q$ be a finite field with $q$ elements and $C \subset \mathbb{F}_{q}^n$ a linear code. As we focus only on linear codes, we omit the word linear from now on. The permutation group of the code $C$ consists of all the permutations $\pi$ of the symmetric group $S_n$ such that $\pi(C) = C$, where $\pi$ acts on $c=(c_1,\ldots,c_n) \in C$ in the natural way as $\pi(c_1,\ldots,c_n) = (c_{\pi(1)},\ldots,c_{\pi(n)})$.

Recently, the permutation groups of codes have attracted a lot of attention due to their implementation in the automorphism ensemble decoding (AED)~\cite{aed-rm, aed-pc, aed-ldpc} and the analysis of capacity-achieving codes for erasure channels~\cite{bec-rm, beyond}.
The AED uses several decoders in parallel, along with some code permutations. However, not every permutation can be used since there are permutations that commute with the decoder; for instance, the lower triangular affine permutations with the successive cancellation decoder for binary polar codes \cite{aed-rm, scinvariantpolar}.

%For a given received word $y$, the AED uses $m$ decoders in parallel to obtain $x_i$ from $\pi_i(y)$, where the $\pi_i$'s are some permutations of the code. As a final output, the AED gives $x=\underset{i}{\mathrm{argmin}}\|y-\pi_i^{-1}(x_i)\|$, the decodification of $y$, where $\|\cdot\|$ is a certain metric \cite{aed-rm}.

%It is important to note that the AED does not work with any subgroup of the permutation group since, for some decoders, certain permutations are invariant, i.e., if $dec$ is the decoding function, $dec(\pi(y))=dec(y)$. This is the case, for instance, for the successive cancellation decoders and the lower triangular affine permutation of binary polar codes \cite{aed-rm, scinvariantpolar}. 

The affine permutation group consists of permutations that depend on an invertible matrix and a vector; see Definition \ref{D:affine-perm}. The affine permutation groups have been studied for their implementation in AED to decode binary polar codes \cite{aed-pc} due to their characterization as monomial codes \cite{appc}. The affine permutation group of polar codes has been completely determined in~\cite{all}.

A monomial code is defined by evaluating certain monomials on a set of points (evaluation points).
Some well-known examples include the Reed-Solomon and the Reed-Muller codes.
When the set of evaluation points $\mathbb{F}_q^m$ is replaced by a Cartesian set in a Reed-Muller code, the evaluation code is called an affine Cartesian code~\cite{CartesianCode}. A monomial Cartesian code is generated by evaluating a fixed set of monomials on a Cartesian set~\cite{MonomialCartesian}. If the set of monomials is closed under divisibility, we call it a decreasing Cartesian code.

%While the definition of decreasing codes requires working with equivalence classes of monomials modulo a vanishing ideal, the problem of finding the automorphism group \ivan{why automorphism group and not affine permutation group?} of a decreasing code uses similar strategies to characterize generic initial ideals. \ivan{Maybe the following is better?}

%Since the definition of decreasing codes requires working with equivalence classes of monomials modulo a vanishing ideal, the problem of finding the affine permutation group of a decreasing code is related to the problem of characterizing generic initial ideals.

Finding the affine permutation group of a decreasing code is equivalent to finding a subgroup of matrices that fixes a monomial set. Thus, it is related to the problem of characterizing generic initial ideals.
Such ideals are invariant under the action of the Borel group of upper triangular non-singular matrices \cite{galligo}. 
These are characterized as Borel ideals \cite{bayer} in characteristic zero and as $p$-Borel ideals \cite{pardue} in positive characteristic. A generalization of this concept is $Q$-Borel ideals \cite{francisco}. %All these ideals can be understood as monomial ideals fixed under the action of certain matrix subgroups.

In this work, we explore the affine permutation group of decreasing Cartesian codes, including the case when the Cartesian set has copies of multiplicative or additive subgroups. This family matters because we can associate a monomial structure to some nonbinary kernels~\cite{vardohus, dec}. However, even for the classical Arikan kernel, the affine permutation group of polar codes depends on the characteristic of the field.

\section{Preliminaries}
%We start by defining monomial codes. We then introduce the concept of permutations and affine permutations.
\subsection{Monomial Cartesian codes}
Let $\mathcal{A}=\prod_{i=1}^m A_i$ be a Cartesian set with $A_i\subseteq\mathbb{F}_q$ and $n_i:=|A_i|\geq 2$. Let $R=\mathbb{F}_q[x_1,\ldots,x_m]$ be the polynomial ring in $m$ variables and denote by $\mathcal{M}$ the monomials of $R$. For any $f\in R$, we define $f(\mathcal{A}) = (f(P_1),\ldots,f(P_n))$, where $\mathcal{A}=\{P_1,\ldots,P_n\}$ with $n=n_1\cdots n_m$.

\begin{definition}\rm
    Fix a set of monomials $ L \subseteq \Delta:=\{u\in\mathcal{M}\ :\ \deg_{x_i} u<n_i\}$. The \textit{monomial Cartesian code}, which depends on the evaluation of the monomials $L$ on the Cartesian set $\mathcal{A}$, is denoted and defined by 
    $$L(\mathcal{A}) = \mathrm{Span}_{\mathbb{F}_q}\left\{f(\mathcal{A})\ :\ f\in L\right\}.$$

The set $L$ is {\it closed under divisibility} if $f$ in $L$ and $g$ a divisor of $f$ implies that $g$ is also in $L$. In this case, we say that the code $L(\mathcal{A})$ is a \textit{decreasing monomial Cartesian code}.
\end{definition}

%The vanishing ideal of $\mathcal{A}$, denoted by $I_\mathcal{A}$, is the set of all polynomials in $R$ that vanish at every point of $\mathcal{A}$. The vanishing ideal plays an important role in defining evaluation codes, since if two polynomials $f$ and $g$ satisfy $f-g \in I_\mathcal{A}$ then \ivan{why not say if and only if?} $f(\mathcal{A})=g(\mathcal{A})$. So, the evaluation code depends only on polynomials modulo the vanishing ideal $I_\mathcal{A}$. In other words, as $I_\mathcal{A}=\left(\prod_{\alpha\in A_j} (x_j-\alpha)\right)_{j=1}^m$ by \cite[Lemma 2.3]{CartesianCode}, for any polynomial $g\in R$, there exists a polynomial $f \in R$ such that $f\in\mathrm{Span}_{\mathbb{F}_q}\Delta$ and $g-f\in I_\mathcal{A}$, meaning $f(\mathcal{A})=g(\mathcal{A})$. We then write $\overline{g}=f$ and obtain that $\mathbb{F}_q^n\cong R/I_\mathcal{A}$ \ivan{this is unclear}. We write $\overline{L}$ to denote the set of all $g\in R$ such that $\overline{g}\in L$. \ivan{How about this:}
The vanishing ideal of $\mathcal{A}$, denoted by $I_\mathcal{A}$, is the set of all polynomials in $R$ that vanish at every point of $\mathcal{A}$. The vanishing ideal plays an important role in defining evaluation codes since for any two polynomials $f$ and $g$ in $R$, we have that $f(\mathcal{A})=g(\mathcal{A})$ if and only if $f-g \in I_\mathcal{A}$. In other words, the evaluation map $f\mapsto f(\mathcal{A})$ induces a linear isomorphism $R/I_\mathcal{A}\cong \mathbb{F}_q^n$. This shows that
the evaluation code depends only on polynomials modulo the vanishing ideal $I_\mathcal{A}$. For a Cartesian set $\mathcal{A}$,
we have $I_\mathcal{A}=\left(\prod_{\alpha\in A_j} (x_j-\alpha)\right)_{j=1}^m$ by \cite[Lemma 2.3]{CartesianCode}. Thus, for any polynomial $g\in R$, there exists a polynomial $f \in\mathrm{Span}_{\mathbb{F}_q}(\Delta)$ with $g-f\in I_\mathcal{A}$ and so $f(\mathcal{A})=g(\mathcal{A})$. We denote such $f$ by $\overline{g}$. Furthermore, for a monomial set $L\subset\Delta$, we use $\overline{L}$ to denote the set of all $g\in R$ such that $\overline{g}\in L$.

\subsection{Affine permutations}
Let $C$ be a code in $\mathbb{F}_q^n$. The \textit{permutation group} of $C$ is denoted and defined by
$$\mathrm{Perm}(C)= \{ \pi\in S_n : \pi(C) = C \},$$
where $\pi$ acts on $c=(c_1,\ldots,c_n) \in C$ as 
%$\pi(c_1,\ldots,c_n) = (c_{\pi(1)},\ldots,c_{\pi(n)})$.
$c_\pi= (c_{\pi(1)},\ldots,c_{\pi(n)})$. We also denote
$\pi(\mathcal{A})=(P_{\pi(1)},\ldots,P_{\pi(n)})$
for a permutation $\pi\in S_n$.

\begin{remark}\rm
Note that for any element $c$ of the evaluation code $L(\mathcal{A})$, there is a polynomial $f\in\mathrm{Span}_{\mathbb{F}_q}(L)$ such that $f(\mathcal{A})=c$. Thus, if $\pi \in \mathrm{Perm}(L(\mathcal{A}))$, then there exists a polynomial $f_\pi$ such that
$$f_\pi(\mathcal{A}) = c_\pi=f(\pi(\mathcal{A})).$$
Then, we can understand $\pi$ as a function on $R/I_\mathcal{A}$, $f\mapsto f_\pi$.
\end{remark}
We are interested in those $\pi$ that can be understood as affine transformations in the following setting.

Let $A$ be an $m\times m$ matrix with entries in $\mathbb{F}_q$ and $b\in\mathbb{F}_q^m$. As usual, the affine transformation $T(x)=Ax+b$ acts on $\mathbb{F}_q^m$ by $T(P)=AP+b,$ where $P=(p_1,\ldots,p_m)^t \in \mathbb{F}_q^m$. But $T(x)$ also acts on $R$ by $$T(f)=f(y_1,\ldots,y_m),$$
where $(y_1,\ldots,y_m)^t=A(x_1,\ldots,x_m)^t+b$. Consequently, $T(x)$ acts on the set of evaluation vectors by
%\[\begin{aligned}
%T(f(\mathcal{A}))=T(f)(\mathcal{A})=&(T(f)(P_1),\ldots, T(f)(P_n))\\
%f(T(\mathcal{A}))=&(f(T(P_1)),\ldots, f(T(P_n))).
%\end{aligned}\]
\[\begin{aligned}
T(f(\mathcal{A}))&=T(f)(\mathcal{A})\\
&=f(T(\mathcal{A})).
\end{aligned}\]
The last two equations lead to the following definition.
\begin{definition}\label{D:affine-perm}\rm
 Let $A$ be an $m\times m$ matrix with entries in $\mathbb{F}_q$ and $b\in\mathbb{F}_q^m$.
 We say that $T$ is an \textit{affine permutation} of $L(\mathcal{A})$ if $T$ leaves invariant $L$ and $\mathcal{A}$; i.e. the following two conditions hold:
\begin{enumerate}
\item[\rm (1)] $T(\mathcal{A})=\mathcal{A}$ and
\item[\rm (2)] $\overline{T(L)} \subseteq \mathrm{Span}_{\mathbb{F}_q}(L)$.
\end{enumerate}
Condition (2) means that for any $f\in L$, $T(f)$ may not be an element of $\mathrm{Span}_{\mathbb{F}_q}(L)$, but $T(f(\mathcal{A}))$ is an element of $L(\mathcal{A})$.
The set of affine permutations of $L(\mathcal{A})$ is denoted by $\mathrm{Perm}_A(L(\mathcal{A}))$.
\end{definition}

The following example shows that condition (1) $T(\mathcal{A})=\mathcal{A}$ is necessary, otherwise, $T$ may not define a permutation.% of the code even if $T(L)\subseteq \mathrm{Span}_{\mathbb{F}_q}(L)$ and $T$ is injective.

\begin{example}\rm
    Take $L=\{x_2,x_1,1\}$, $\mathcal{A}=\mathbb{F}_3^\ast\times\{0,1\}=\{(1,0),(1,1),(2,0),(2,1)\}$, and $T(x)=\begin{pmatrix} 1&0\\ 1&1\end{pmatrix}x$. We have
    $$T(\mathcal{A})=\{(1,1),(1,2),(2,2),(2,0)\}\neq\mathcal{A}$$
    $$\text{and } T(f(x_1,x_2))=f(x_1,x_1+x_2).$$

    For $f(x_1,x_2)=x_2-x_1+1$, we have $T(f(x_1,x_2))=x_2+1$. Thus,    
     $f(\mathcal{A})=(0,1,2,0)$ and $T(f(\mathcal{A}))=(1,2,1,2)$, meaning that $T$ does not even define an isometry of the code $L(\mathcal{A})$.
\end{example}

\subsection{Borel movements}
    Let $u$ be a monomial in $R=\mathbb{F}_q[x_1,\ldots,x_m]$. If the indeterminate $x_i$ divides $u$, and $j<i$, the monomial $\frac{x_j}{x_i}u$ is called a \textit{Borel movement} of $u$.
  
    We say that a monomial set $L$ satisfies the \textit{Borel property} if $L$ is closed under Borel movements; i.e., if $u$ is a monomial of $L$, then any Borel movement of $u$ is also in $L$. In this case, we say that the monomial code $L(\mathcal{A})$ has the \textit{Borel property}.
    
        Let $p=\mathrm{char}(\mathbb{F}_q)$. For any $m,n\in\mathbb{N}$, we write $m\leq_p n$ if and only if $m_k\leq n_k$ for all $k\in\mathbb{N}$, where $m=\sum_{k=0}^\infty m_kp^k$ and $n=\sum_{k=0}^\infty n_kp^k$ are the $p$-adic expansions.

        Let $u$ be a monomial in $R$. If the indeterminate $x_i$ divides $u$, $\ell\leq_p \deg_{x_i}u$, and $j<i$, then the monomial $\left(\frac{x_j}{x_i}\right)^\ell u$ is called a \textit{standard $p$-Borel movement} of $u$. 

\section{All the points}
This manuscript aims to describe the affine permutation group for certain monomial Cartesian codes $L(\mathcal{A})$. In this section, we study the case when $\mathcal{A} = \mathbb{F}_q^m$. This family of codes covers, for instance, the Reed-Muller codes.

\begin{example}\rm
The affine permutation group of the Reed-Muller codes is the set of all bijective affine transformations \cite{berger}.
\end{example}

In \cite{appc}, the authors proved that a polar code is a decreasing monomial code $L(\mathcal{A})$ where $L$ has the Borel property and $\mathcal{A}=\mathbb{F}_2^m$.

\begin{remark}\rm
    In \cite{appc}, a monomial set closed under divisibility is called weakly decreasing. A weakly decreasing set with the Borel property is called decreasing. Here, we use the term Borel property in analogy to the property satisfied by Borel ideals in characteristic zero \cite{bayer}. 
\end{remark}

The \textit{lower triangular affine transformations} are the transformations $T=Ax+b$ with an invertible lower triangular matrix $A$. We denote the subgroup of lower triangular affine transformations by $\mathrm{LTA}_m$.

% We come to one of the main results of this section, where we prove that the group of affine permutations of a decreasing code with the Borel property contains the lower triangular affine transformations. \ivan{This paragraph is too similar to the one before Theorem 11. I think this is not as strong as Th 11, so maybe rephrase this?}

Similar to the binary case, if $L$ has the Borel property, the $L(\mathbb{F}_q^m)$ contains the lower triangular affine transformations. We generalize this result in theorem \ref{nice theorem}.

\begin{theorem}\label{24.01.20}
 If $L(\mathbb{F}_q^m)$ is a decreasing code with the Borel property, then
 $$\mathrm{LTA}_m\subseteq\mathrm{Perm}_A(L(\mathbb{F}_q^m)).$$
\end{theorem}

\begin{proof}
    Let $T=Ax+b$ an element in $\mathrm{LTA}_m$.
    As $\det A\neq 0$, $T$ is an automorphism of $\mathbb{F}_q^n$, meaning that $T(\mathbb{F}_q^n)=\mathbb{F}_q^n$.
    
    Let $x^\nu$ be a monomial in $L$. Note that $$T(x_i) = \sum_{j=1}^i A_{ij}x_j+b_i.$$ Thus, $T(x^\nu)$ is a polynomial supported on the Borel movements of the divisors of $x^\nu$. Since $L$ is decreasing and satisfies the Borel property, $T( x^\nu)\in\mathrm{Span}_{\mathbb{F}_q}(L)$.
\end{proof}

We can extend Theorem~\ref{24.01.20} to any decreasing code $L(\mathbb{F}_q^m)$ without the Borel property by looking at the pairs $(x_i,x_j)$ such that if $u$ is an element in $L$, and $x_i$ divides $u$, then the monomial $\frac{x_j}{x_i}u$ is also an element in $L$. Defining a set of matrices with the $(i,j)$ entry equals zero if $(x_i, x_j)$ is not such a pair, we can obtain a subgroup of affine transformations fixing the code. However, due to the characteristic of the field, sets of monomials with the Borel property are not the only ones that the action of such matrices can fix, as the following example shows. %\ivan{i am not sure it's clear what we mean by "more sets of monomials". Would it be better to say "larger monomial sets"? I feel like it's kind of an inside-out logic. In the example below we have $L$ without the Borel property and the corresponding matrix subgroup is still lower triangular}

\begin{example}\rm
    Consider the following sets of monomials in $\mathbb{F}_9[x_1,x_2]$: $L_1$ is the set of monomials of degree at most 3,
    $L_2=\{x_2^4,x_1x_2^3,x_1^3x_2,x_1^4\}$, and $L = L_1 \cup L_2$.
    Take the affine transformation $T(x)=\begin{pmatrix} a&0\\  b&c\end{pmatrix}x$.
    Observe $T(L_1) \subseteq \mathrm{Span}_{\mathbb{F}_9}(L_1)$ by Theorem~\ref{24.01.20} and because $L_1$ has the Borel property.
    
    As $T(x_1)=ax_1$ and $T(x_2)=bx_1+cx_2$, we have
    $$\begin{array}{ccc}
        T(x_2^4)&=&b^4\ x_1^4\ +\ b^3c\ x_2x_1^3\ +\ bc^3\ x_2^3x_1\ +\ c^4\ x_2^4\\
        T(x_1x_2^3)&=&ab^3\ x_1^4\ +\ ac^3\ x_1x_2^3\\
        T(x_1^3x_2)&=&a^3b\ x_1^4\ +\ a^3c\ x_1^3x_2\\
        T(x_1^4)&=&a^4x_1^4.
    \end{array}$$
    Thus, $T(L_2) \subseteq \mathrm{Span}_{\mathbb{F}_9}(L_2)$. Note that the characteristic of the field $\mathbb{F}_9$ matters when computing $T(x_2^4)$, $T(x_1x_2^3)$ and $ T(x_1^3x_2)$. %\ivan{it matters for the 2nd and 3rd computations too, isn't it?}

The monomial $x_1^2x_2^2$ is a Borel movement of $x_1x_2^3 \in L$, but $x_1^2x_2^2$ is not in $L$. So, $L$ does not have the Borel property.

We conclude that any lower triangular affine transformation $T$ fixes $L$, meaning $T(L) \subseteq \mathrm{Span}_{\mathbb{F}_9}(L)$, even when $L$ does not have the Borel property.
\end{example}

We capture some of those sets using the $p$-Borel movement.

\begin{definition}\rm
    For a monomial set $L \subset \mathbb{F}_q[x_1,\ldots,x_m]$, we define the $p$-Borel graph of $L$ as the directed graph $G_L$ with the vertex set $\{x_1,\ldots,x_m\}$ and the edge set $E(G_L)$ where $(x_i,x_j)\in E(G_L)$ if and only if for any $u\in L$ divisible by $x_i$ the monomial $\left(\frac{x_j}{x_i}\right)^\ell u$ lies in $L$ for all $0\leq\ell\leq_p \deg_{x_i}(u)$.

    For $u\in L$, if $u'$ can be obtained through a sequence of $p$-Borel movements involving pairs $(x_i,x_j)\in E(G_L)$, we say that $u'$ is a valid $p$-Borel movement of $u$ with respect to $L$.
\end{definition}
%\ivan{there were a few typos in the definition, so I rewrote it. Please check if you like it}

\begin{definition}\rm
    Let $G_L$ be the $p$-Borel graph of the monomial set $L$. We define the space of $L$-stable matrices as 
    $$M_L=\{A\in\mathbb{F}_q^{m\times m}\ :\ A_{ij}=0\ \text{if}\ (x_i,x_j)\notin E(G_L)\}.$$
\end{definition}

We come to one of the main results of this section, where we prove that the group of affine permutations of a decreasing code contains the affine permutations that depend on the invertible $L$-stable matrices.
\begin{theorem}\label{nice theorem}
    If $L(\mathbb{F}_q^m)$ is a decreasing code, then
    $$\{Ax+b : A \text{ invertible}, A\in M_L\}\subseteq\mathrm{Perm}_{A}(L(\mathbb{F}_q^m)).$$
\end{theorem}

\begin{proof}
Let $A=(a_{ij})_{ij=1}^m\in M_L$ be invertible and
let $T(x)=Ax+b$ be the corresponding affine transformation, for some $b\in\mathbb{F}_q^m$. Clearly, 
$T(\mathbb{F}_q^n)=\mathbb{F}_q^n$.

%\ivan{How about we start it as you wrote below. Then comment everything starting from your comment and to the end of the proof}

Let $u=x_1^{v_1}\cdots x_m^{v_m}\in L$. We have $T(u)=y_1^{v_1}\cdots y_m^{v_m}$, where 
$y_i=\sum_{j=1}^m a_{ij} x_j+b_i$. Consider
$w\in\mathrm{supp}(T(u))$. Then $w=u_1\cdots u_m$ where each $u_i$ is in the support of $y_i^{v_i}$. We will show below that each $u_i$ corresponds to a sequence of valid $p$-Borel movements of  $x_i^{v'_i}$ for some $v'_i\leq v_i$. This implies that $w$ is realized through a sequence of valid $p$-Borel movements of a divisor of $u$, and thus $w\in L$, which implies that $T(u)$ lies in the span of $L$.

To show the claim, let us compute the support of $y_i^v$ for some $1\leq v< q$. We have
    \begin{align*}
    y_i^v&=\left(\sum_{j=1}^m a_{ij} x_j+b_i\right)^v\\&=\sum_{\substack{ k_0+\ldots+k_m=v,\\ k_s\geq 0}}{v\choose k_0,\ldots,k_m} b_i^{k_0}\prod_{j=1}^m (a_{ij}x_j)^{k_j}.
    \end{align*}

    It is known that $p$ does not divide ${v-\sum_{s=0}^{t-1} k_s \choose k_{t}}$ if and only if $k_{t}{\leq}_p v-\sum_{s=0}^{t-1} k_s$ (cf. \cite[Corollary II.3]{pardue}).
    If a monomial $u'$ appears in the support of $y_i^v$, then there are $k_0,\dots, k_m$ with $k_0+k_1+\ldots+k_m=v$ such that
    \begin{equation}\label{eq prop}
    u'=(x_i)^{v-k_0}\left(\frac{x_1}{x_i}\right)^{k_1}\cdots\left(\frac{x_m}{x_i}\right)^{k_m},
    \end{equation}
    where each $k_t\leq_p v-\sum_{\substack{s=0\\ s\neq i}}^{t-1} k_s$. Thus, if $k_t\neq 0$, then $(x_i,x_t)\in E(G_L)$ or $i=t$. Thus, the monomials in the support of $y_i^v$ are valid $p$-Borel movements.
    \end{proof}

\section{Multiplicative subgroups}
We now consider the case when every $A_i$ of the Cartesian set $\mathcal{A}=\prod_{i=1}^m A_i$ is either $\mathbb{F}_q$ or a subgroup of $\mathbb{F}_q^\ast$. This family of evaluation codes contains, for example, the Reed-Solomon codes with $\mathbb{F}_q^*$ as the evaluation set and the well-known family of toric codes whose evaluation set is the torus \cite{toric}.

Let $\mathbb{F}_q^*=\langle \beta \rangle = \langle 1, \beta, \ldots, \beta^{q-2} \rangle $. Any proper subgroup of $\mathbb{F}_q^*$ of size $s$ has the form $G=\langle \beta ^t\rangle$, with  $1<t$, $t|q-1$, and $|G|=s=\frac{q-1}{t}$.
\begin{remark}\label{24.01.21}
Note that $G=\{x\in\mathbb{F}_q\ :\ x^s-1=0\}$. In particular, the sum of the elements of $G$ is zero when $|G|>1$.
\end{remark}

\begin{lemma}\label{L0}
Let $G_1$ and $G_2$ be nontrivial (not necessarily distinct) subgroups of $\mathbb{F}_q^\ast$. For any $a \in\mathbb{F}_q$ and $b  \in \mathbb{F}_q^*$, we have
$$
a  G_1+b  \, \neq \, G_2.
$$
\end{lemma}

\begin{proof}

Assume $a  G_1+b   = G_2$. Then, the sum of the elements of $a  G_1+b  $ equals the sum of the elements of $G_2$, which is $0$ by Remark~\ref{24.01.21}. Then
$$
0=\sum_{g\in G_1} (a  g+b  )=a  \sum_{g\in G_1} g+sb  =sb  .
$$
Since $b  \neq 0$ and $\mathrm{char}(\mathbb{F}_q)\!\not{|}\,s$, we get a contradiction.
\end{proof}

\begin{lemma} \label{LA}
%Assume $2 \leq r \leq m$ and 
Let $\mathcal{A}=\prod_{i=1}^m G_i$, where every $G_i$ is a nontrivial subgroup of $\mathbb{F}_q^{\ast}$.
For any $1\leq i\leq m$, $a\in\mathbb{F}_q^m$ of weight at least two, and $b \in\mathbb{F}_q$, there exists $g$ in $\mathcal{A}$ such that $$a\cdot g+b \notin G_i,$$ where $\cdot$ represents the standard Euclidean inner product. 

% Let $m\geq 2$  and $q\geq 3$. For any $a\in \mathbb T^m$ there exists $\mathbf{g}\in G^m$ such that $a\cdot \mathbf{g}\notin G$.
\end{lemma}

%\ivan{What does the dot product mean if $r$ is less than $m$? }

\begin{proof}
%Without loss of generality, we can assume $i=1$. 
Without loss of generality, we may assume that $a_1,a_2\neq 0$, where $a=(a_1,\ldots, a_m)$. 
If $a_2+\cdots+a_m+b \neq 0$, we let $g_2=1$; otherwise, we let $g_2$ be an arbitrary non-unit element of $G_2$.
%Define $g_2=1 \in G_2$ if , or $g_2\neq 1$ as any element in $G_2$ otherwise. 
Then $d=g_2a_2+a_3+\cdots+a_m+b \neq 0$.
By Lemma~\ref{L0}, there exists $g_1 \in G_1$ such that $a_1g_1+d\notin G_i$. Therefore, we can take $g=(g_1,g_2,1,\dots,1)$ in $\mathcal{A}$.
%By Lemma~\ref{L0}, there exists $g_1 \in G_1$ such that $a\cdot g=a_1g_1+x\notin G_1$ where $g=(g_1,g_2,1,\dots,1)\in\mathcal{A}$.
\end{proof}

% \begin{lemma} \label{LB} Let $k\geq 2$  and $q\geq 3$. For any $a\in \mathbb T^k$ and $\lambda\in\mathbb F_q^*$ there exists $\mathbf{g}\in \mathbb G^k$ such that $a\cdot \mathbf{g}+\lambda \notin G$.
% \end{lemma}

% \begin{proof}
% Let $a=(a_1,a_2,\ldots, a_k)$ and $\lambda\neq 0$. We can assume  $x=a_2+\cdots+ a_k+\lambda\neq 0$, then for  $a_1=\beta^i$, there exist $g_1 \in G$ such that $\beta^ig_1+x\notin G$. Let $\mathbf{g}=(g_1,1,\dots,1)$ and
% $$
% a\cdot \mathbf{g}=a_1g_1+a_2+\cdots+ a_k+\lambda=\beta^ig+x\notin G
% $$
% \end{proof}

\begin{proposition}\label{Prop:multisub}
Assume $\mathcal{A}=\prod_{i=1}^m G_i$, where every $G_i$ is a nontrivial subgroup of $\mathbb{F}_q^{\ast}$.
Then, an affine transformation $T(x)=Ax+b$ satisfies $T(\mathcal{A})=\mathcal{A}$ if and only if $b=0$ and $A=P_{\sigma}D$, for a permutation matrix $P_\sigma$ and a diagonal matrix $D$ such that
    $G_{\sigma(i)}=G_i$ and $D_{ii}\in G_i$ for all $1\leq i\leq m$.  
\end{proposition}
%\ivan{Is $\mathcal{G}$ supposed to be $\mathcal{A}$ or the other way around?}
\begin{proof}
The ``if'' part is clear.
Let $T(x)=Ax+b$ be an affine transformation such that $T(\mathcal{A})=\mathcal{A}$. We claim that each row of $A$ has exactly one nonzero entry. Clearly, $A$ cannot have zero rows. Let $a$ be a row of $A$ with a weight of at least two. By Lemma \ref{LA}, there is $g\in\mathcal{A}$ such that $a\cdot g+b_i\notin G_i$, where $i$ is the position of the row $a$ in $A$. Thus, $T(\mathcal{A})\neq \mathcal{A}$ and we have a contradiction.

Now, since the weight of each row of $A$ is one, Lemma \ref{L0} implies that $b=0$ and so $T(x)=Ax$. Let $a_{ij}$ be the only nonzero entry of the $i$-th row of $A$. Then $a_{ij}G_j=G_i$. But then $G_i=G_j$ and $a_{ij}\in G_i$, so we have the conclusion.
\end{proof}

Let $\mathbb T^m=({\mathbb F}_q^*)^m$ be the $m$-dimensional algebraic torus.

\begin{corollary}
     An affine transformation $T(x)=Ax+b$ satisfies $T(\mathbb T^m)=\mathbb T^m$ if and only if $b=0$ and $A=PD$, where $P$ is a permutation matrix and $D$ is a nonsingular diagonal matrix.  
\end{corollary}

We now consider the Cartesian set $\mathcal{A}=\prod_{i=1}^m A_i$, where for every $1\leq i\leq m$ either $A_i=\mathbb{F}_q$ or $A_i=\mathbb{F}_q^\ast$. Without loss of generality, we may assume $\mathcal{A}=\mathbb{F}_q^s\times(\mathbb{F}_q^\ast)^{m-s}$ since the permutation of variables is an affine transformation corresponding to a permutation matrix.

\begin{proposition}\label{24.01.24}
    Assume $\mathcal{A}=\mathbb{F}_q^s\times(\mathbb{F}_q^\ast)^{m-s}$ and let $T(x)=Ax+b$ be an affine transformation. We have $T(\mathcal{A})=\mathcal{A}$ if and only if $b_i=0$ for $i>s$ and
    $$A=\begin{pmatrix}
    A_1& A_2\\
    0&PD\end{pmatrix},$$
    \noindent where $P$ is an $(m-s)\times (m-s)$ permutation matrix, $D$ a nonsingular diagonal matrix, and $A_1$ is an $s\times s$ nonsingular matrix.
\end{proposition}

\begin{proof}
 Consider $i>s$. Lemma \ref{LA} implies that there is $g\in\mathbb{T}^m$ such that $(Ag+b)_i=0$, unless the weight of the $i$-th row of $A$ is $1$ and $b_i=0$. Thus, $(Ax+b)_i=a_jx_j$ for some $a_j\in\mathbb{F}_q^*$. Also, $j>s$, otherwise for any point $v\in\mathcal{A}$ with $v_j=0$ we have $(Av+b)_i=0$ and thus $Av+b\notin\mathcal{A}$. This proves that the bottom $m-s$ rows of $A$ form a block $(0\ \ PD)$ for some permutation matrix $P$ and some nonsingular diagonal matrix~$D$.

%We claim that if $$A=\begin{pmatrix}
%    A_1& A_2\\
%    0&PD\end{pmatrix},$$ $Ax+b$ is a bijection of $\mathcal{A}$ if and only if $A_1$ is nonsingular.

    If $A_1$ is singular, there is nonzero $v\in\mathbb{F}_q^s$ such that $A_1v=0$. Let $w=(v,{1,\ldots,1})\in\mathbb{F}_q^m$ and 
    $w'=(0,\ldots,0,1,\ldots,1)\in\mathbb{F}_q^m$, with $wt(w')=m-s.$ We have $w,w'\in\mathcal{A}$ and $Aw+b=Aw'+b$, thus $Ax+b$ is not an injection. The other direction is trivial.
\end{proof}

\begin{corollary}\label{24.01.24-2}
    Take $\mathcal{A}=\mathbb{F}_q^s\times(\mathbb{F}_q^\ast)^{m-s}$.
    If $L(\mathcal{A})$ is a decreasing code with the Borel property, then
    an affine transformation $T(x)=Ax+b$ lies in $\mathrm{Perm}_A(L(\mathcal{A}))$ if
    $$A=\begin{pmatrix}
    A_1&0\\
    0&I_{m-s}\end{pmatrix},$$
    where $A_1$ is a lower triangular matrix, $I_{m-s}$ is the identity of size $m-s$, and $b_i=0$ for $i>s$.
\end{corollary}

\begin{proof}
    Since $L$ has the Borel property, then $L$ is stabilized by lower triangular affine transformations. The matrices in Proposition \ref{24.01.24} that are lower triangular are precisely those of the given shape.
\end{proof}

\begin{proposition}\label{24.01.24-1}
Take $\mathcal{A}=\mathbb{F}_q^{m_0}\times \prod_{i=1}^l G_i^{m_i}$, where the $G_i$'s are distinct subgroups of $\mathbb{F}_q^\ast$. Then, an affine transformation $T(x)=Ax+b$ fixes $\mathcal{A}$ if and only if $b_i=0$ for $i>m_0$ and
    $$A=\begin{pmatrix}
    A_{0}& A_{1} & \ldots&  A_{l}\\
    0& P_1D_{1} & \ldots&  0\\
%    0& 0 & P_2D_{2} & \ldots&  0\\    
   \ldots& \ldots & \ldots& \ldots\\ 
   0& 0 & \ldots& P_l D_l\\ 
    \end{pmatrix},$$
    \noindent where $P_i$ is an $m_i\times m_i$ permutation matrix, $D_i$ a nonsingular diagonal matrix with entries in $G_i$, and $A_0$ is an $m_0\times m_0$ nonsingular matrix.
\end{proposition}

\begin{proof} The ``if" direction is clear.
Consider $T(x)=Ax+b$. Let $A'$ be the submatrix of $A$ obtained by removing the first $m_0$ rows and the first $m_0$ columns. Let $b'$ be the vector obtained by removing the first $m_0$ entries of $b$. Then $A'x+b'$ is a bijective affine transformation of $G_1^{m_1}\times\cdots\times G_l^{m_l}$ and by Proposition \ref{Prop:multisub} we have that $b'=0$ and $A'$ has the desired form. Lemma \ref{LA} proves that for $i>m_0$ and $j\leq m_0$, $A_{ij}=0$. Thus, $A$ has the desired property.
%    $$A=\begin{pmatrix}
%    A_{0}& A_{1} & A_{2} & \ldots&  A_{l}\\
%    0& P_1D_{1} & 0 & \ldots&  0\\
%    0& 0 & P_2D_{2} & \ldots&  0\\    
%   \ldots& \ldots & \ldots & \ddots& \ldots\\ 
%   0& 0 & 0 & \ldots& P_l D_l\\ 
%    \end{pmatrix}.$$
Finally, since $A$ is nonsingular, $A_0$ is also nonsingular.
%$A$ is otherwise there is $v,w\in\mathbb{F}_q^{m_0}$ such that $T(v,1\ldots,1)=T(w,1,\ldots,1)$, which is a contradiction. The other implication is trivial.
\end{proof}

%We come to the main result of this section.
\begin{theorem}
    Let $\mathcal{A}=\mathbb{F}_q^{m_0}\times \prod_{i=1}^l G_i^{m_i}$, where the $G_i$'s are distinct subgroups of $\mathbb{F}_q^\ast$. If $L(\mathcal{A})$ is a decreasing code with the Borel property, then $T(x)=Ax+b$ lies in $\mathrm{Perm}_A(L(\mathcal{A}))$ if $b_i=0$ for all $i>m_0$ and
     $$A=\begin{pmatrix}
    A_{0}& 0 & \ldots&  0\\
    0& I_{m_1} & \ldots&  0\\
%    0& 0 & I_{m_2} & \ldots&  0\\
    \ldots& \ldots & \ddots& \ldots\\
   0& 0& \ldots &I_{m_l}\\ 
    \end{pmatrix}.$$
\end{theorem}

\begin{proof}
    The proof is similar to the one of Corollary \ref{24.01.24-2}. It follows from Proposition \ref{24.01.24-1} and the fact that lower triangular affine transformations stabilize $L$.
\end{proof}

\section{Additive groups}

We now consider additive subgroups of $\mathbb{F}_q$ to build Cartesian sets. While in the case of multiplicative subgroups, the affine permutation group is heavily reduced to a subgroup of permutation matrices, additive subgroups can still have a richer structure in their automorphism group. However, the choice of points still imposes several limitations on the matrices.

Recall that if $G$ is an additive subgroup of $\mathbb{F}_q$, then $G$ is a vector space over $\mathbb{F}_p$, where $p$ is the characteristic of $\mathbb{F}_q$. We now prove which possible affine transformations preserve $G$.

\begin{proposition}\label{Prop.add}
    Let $G$ be an additive subgroup of $\mathbb{F}_q$
    and $T(x)=ax+b$, $a,b\in\mathbb{F}_q$, a bijection of $G$. Then $b\in G$ and $a\in\mathbb{F}_{q'}$, where $\mathbb{F}_{q'}$ is the largest subfield of $\mathbb{F}_q$ such that $G$ is an $\mathbb{F}_{q'}$-vector subspace.
\end{proposition}

\begin{proof}
    We have $T(0)=b\in G$ since $T$ is a bijection. Since the effect of $b$ is trivial, we can assume that $T(x)=ax$. It can be readily seen that the set $H=\{a\in\mathbb{F}_q : aG=G\}$ is a subring and, hence, a subfield of $\mathbb{F}_q$. Therefore, $H=\mathbb{F}_{q'}$ for some $q'|q$, as required.
\end{proof}

\begin{example}\rm
    Let $\alpha$ be a primitive element of $\mathbb{F}_{16}$ with $\alpha^4+\alpha+1=0$. Let $G=\alpha^6\mathbb{F}_2+\alpha^{11}\mathbb{F}_2$. Note that $G$ is a vector subspace of dimension $2$ over $\mathbb{F}_2$ and of dimension $1$ over $\mathbb{F}_4$. In fact, $G=\alpha\mathbb{F}_4$, and so $T(x)=ax+b$ is a bijection of $G$ if and only if $b\in G$ and $a\in\mathbb{F}_4$.
\end{example}

We can characterize the affine permutations preserving $G^m$ for an additive subgroup $G$ of $\mathbb{F}_q$. 

\begin{corollary}\label{25.01.24}
    Let $\mathcal{A}=G^m$, where $G$ is an additive subgroup of $\mathbb{F}_q$. Let $\mathbb{F}_{q'}$ be the largest subfield of $\mathbb{F}_q$ such that $G$ is an $\mathbb{F}_{q'}$-vector space. Then an affine transformation $T(x)=Ax+b$ fixes $\mathcal{A}$ if and only if $b\in\mathcal{A}$ and $A$ is a nonsingular matrix over $\mathbb{F}_{q'}$.
\end{corollary}
%\ivan{Why do we need $G$ to be proper?}

\begin{proof}
    Let $T(x)=Ax+b$ be an affine transformation that fixes $\mathcal{A}$. Since $0\in\mathcal{A}$, we have $b\in\mathcal{A}$ and so we can assume that $b=0$.

    By Proposition \ref{Prop.add}, $A$ should have entries in $\mathbb{F}_{q'}$. $A$ should be nonsingular, otherwise there exists $0\neq v\in\mathbb{F}_{q'}^m$ such that $Av=0$. If $0\neq\alpha\in G$, then $\alpha v\in\mathcal{A}^m$ and $T$ is not a bijection. Thus $A$ should be nonsingular over $\mathbb{F}_{q'}$.
\end{proof}

With this, we can characterize some of the affine permutations of Cartesian codes evaluated over $G^m$ for some additive subgroup $G$.

\begin{theorem}\label{T:same-G}
    Let $\mathcal{A}=G^m$, where $G$ is an additive subgroup of $\mathbb{F}_q$. Let $\mathbb{F}_{q'}$ be the largest subfield of $\mathbb{F}_q$ such that $G$ is an $\mathbb{F}_{q'}$-vector space. If $L(\mathcal{A})$ is a decreasing code with the Borel property, then $T(x)=Ax+b$ lies in $\mathrm{Perm}_A(L(\mathcal{A}))$ if $b\in G$ and $A$ is a nonsimgular lower triangular matrix over $\mathbb{F}_{q'}$.
\end{theorem}

\begin{proof}
    Since $L$ has the Borel property, it is stabilized by lower triangular affine transformations. By Corollary \ref{25.01.24}, we have the conclusion.
\end{proof}

In the case when $\mathcal{A}=\prod_{i=1}^m G_i$ where the $G_i$'s are not necessarily different additive subgroups of $\mathbb{F}_q$,
the answer is not as elegant as in Theorem \ref{T:same-G} above.

Let $T(x)=Ax+b$ be an affine transformation that fixes $\mathcal{A}$. As before, since $0\in\mathcal{A}$, we have $b\in\mathcal{A}$, and so we can assume that $b=0$. Let $v\in\mathcal{A}$ be an element with just one nonzero entry in position $i$. Then $T(v)=v_i A_i$, where $A_i$ is the $i$-th column of $A$. 
Since $v_iA_i\in\mathcal{A}=\prod_{i=1}^m G_i$, then
\begin{equation}\label{24.01.29}
A_{ij}\in H_{ij}:=\{a\in\mathbb{F}_q\ |\ aG_i\subseteq G_j\}.   
\end{equation}

For the case where $G_i=G_j$, we now that $H_{ij}$ is the biggest subfield $\mathbb{F}_{q'}$ of $\mathbb{F}_q$ such that $G_i$ is an $\mathbb{F}_{q'}$ vector-space. If $|G_i|>|G_j|$, then $H_{ij}=\{0\}$. If $|G_i|\leq |G_j|$, $H_{ij}$ is an additive subgroup, but it is not necessarily closed under products, and thus, it is no longer a field. Even in the case where $G_i\subsetneq G_j$ and $G_j$ is a $\mathbb{F}_{q'}$-vector space, $H_{ij}$ can be bigger than $\mathbb{F}_{q'}$.

\begin{example}\label{24.01.24-ex1}\rm
    In $\mathbb{F}_{16}$, let $\alpha$ be a primitive element with $\alpha^4+\alpha+1=0$. Let $G_1=\mathbb{F}_2+\alpha\mathbb{F}_2+\alpha^2\mathbb{F}_2$ and $G_2=\alpha\mathbb{F}_4$ and $G_3=\mathbb{F}_2$. Then
    $$\begin{array}{lll}
    H_{11}=\mathbb{F}_2& H_{12}=\alpha^{-1}\mathbb{F}_4& H_{13}=G_1\\
    H_{21}=\{0\}& H_{22}=\mathbb{F}_4& H_{23}=G_2\\
    H_{31}=\{0\}& H_{32}=\{0\}& H_{33}=G_3
    \end{array},$$
where $H_{ij}$ is defined in~(\ref{24.01.29}).
Despite the matrices $A\in\mathbb{F}_q^{m\times m}$ such that $A_{ij}\in H_{ij}$ are not necessarily trivial, an incompatible structure of a monomial set $L$ may impose extra conditions that reduces the affine permutation to a trivial one.
\end{example}

\begin{example}\rm
    Let $G_1,G_2,G_3\subset\mathbb{F}_{16}$ as in Example \ref{24.01.24-ex1}. Let $L$ be the set of divisors of the monomials in $\{x_1^2,x_1x_2\}$.

   Observe that $L$ has the Borel property. A matrix $A\in\mathbb{F}_q^{3\times 3}$, with $A_{ij}\in H_{ij}$ defined in~(\ref{24.01.29}), is an upper triangular matrix for any $1\leq i,j\leq 3$. Thus,
   $A=\begin{pmatrix} a&b&c\\ 0&d&e\\0&0&f\end{pmatrix}.$
   If $T=Ax$,
   $$T(x_1^3)=a^2x_1^2+abx_1x_2+acx_1x_3+bcx_2x_3+b^2x_2^2+c^2x_3^2.$$
  The last four terms are not in $L$, so $b=c=0$. Analogously, $e=0$. Thus, the only affine permutations of $L(\mathcal{A})$, where $\mathcal{A}=G_1\times G_2\times G_3$, are $T=x+b$ for any $b\in\mathcal{A}$. 
\end{example}

\section{Conclusion}
An evaluation code depends on the evaluation of certain monomials at some points. An affine transformation $T(x)=Ax+b$ defines a permutation of an evaluation code if the sets of monomials and points are invariant under the action of $T$.
This paper studies the affine permutations of monomial Cartesian codes when the Cartesian set has copies of multiplicative or additive subgroups. This family of codes includes, in particular cases, the Reed-Muller and the Reed-Solomon codes. When the set of monomials is decreasing (closed under divisibility) or has the Borel property (closed under Borel movements), we provide the conditions for $A$ and $b$ to determine if $T$ defines a permutation. Our findings give insight into studying the automorphism of polar codes that are associated with Reed-Solomon kernels and can be seen as decreasing Cartesian codes.
\bibliographystyle{abbrv}
\bibliography{ref}

\begin{thebibliography}{10}

\bibitem{appc}
M.~Bardet, V.~Dragoi, A.~Otmani, and J.-P. Tillich.
\newblock Algebraic properties of polar codes from a new polynomial formalism.
\newblock In {\em 2016 IEEE International Symposium on Information Theory (ISIT)}, pages 230--234. IEEE, 2016.

\bibitem{bayer}
D.~A. Bayer.
\newblock {\em The division algorithm and the Hilbert scheme}.
\newblock Harvard University, 1982.

\bibitem{berger}
T.~Berger and P.~Charpin.
\newblock The automorphism group of generalized reed-muller codes.
\newblock {\em Discrete mathematics}, 117(1-3):1--17, 1993.

\bibitem{dec}
E.~Camps, H.~H. L{\'o}pez, G.~L. Matthews, and E.~Sarmiento.
\newblock Polar decreasing monomial-cartesian codes.
\newblock {\em IEEE Transactions on Information Theory}, 67(6):3664--3674, 2020.

\bibitem{vardohus}
E.~Camps~Moreno, E.~Martínez-Moro, and E.~Sarmiento~Rosales.
\newblock Vardøhus codes: Polar codes based on castle curves kernels.
\newblock {\em IEEE Transactions on Information Theory}, 66(2):1007--1022, 2020.

\bibitem{francisco}
C.~A. Francisco, J.~Mermin, and J.~Schweig.
\newblock Generalizing the borel property.
\newblock {\em Journal of the London Mathematical Society}, 87(3):724--740, 2013.

\bibitem{galligo}
A.~Galligo.
\newblock A propos du th{\'e}oreme de pr{\'e}paration de weierstrass.
\newblock In {\em Fonctions de Plusieurs Variables Complexes: S{\'e}minaire Fran{\c{c}}ois Norguet Octobre 1970--D{\'e}cembre 1973}, pages 543--579. Springer, 2006.

\bibitem{aed-ldpc}
M.~Geiselhart, M.~Ebada, A.~Elkelesh, J.~Clausius, and S.~ten Brink.
\newblock Automorphism ensemble decoding of quasi-cyclic ldpc codes by breaking graph symmetries.
\newblock {\em IEEE Communications Letters}, 26(8):1705--1709, 2022.

\bibitem{aed-rm}
M.~Geiselhart, A.~Elkelesh, M.~Ebada, S.~Cammerer, and S.~ten Brink.
\newblock Automorphism ensemble decoding of reed--muller codes.
\newblock {\em IEEE Transactions on Communications}, 69(10):6424--6438, 2021.

\bibitem{all}
M.~Geiselhart, A.~Elkelesh, M.~Ebada, S.~Cammerer, and S.~ten Brink.
\newblock On the automorphism group of polar codes.
\newblock In {\em 2021 IEEE International Symposium on Information Theory (ISIT)}, pages 1230--1235. IEEE, 2021.

\bibitem{toric}
J.~P. Hansen.
\newblock Toric surfaces and error-correcting codes.
\newblock In {\em Coding Theory, Cryptography and Related Areas: Proceedings of an International Conference on Coding Theory, Cryptography and Related Areas, held in Guanajuato, Mexico, in April 1998}, pages 132--142. Springer, 2000.

\bibitem{bec-rm}
S.~Kudekar, S.~Kumar, M.~Mondelli, H.~D. Pfister, E.~{\c{S}}a{\c{s}}o{\u{g}}lu, and R.~Urbanke.
\newblock Reed-muller codes achieve capacity on erasure channels.
\newblock In {\em Proceedings of the forty-eighth annual ACM symposium on Theory of Computing}, pages 658--669, 2016.

\bibitem{beyond}
S.~Kumar, R.~Calderbank, and H.~D. Pfister.
\newblock Beyond double transitivity: Capacity-achieving cyclic codes on erasure channels.
\newblock In {\em 2016 IEEE Information Theory Workshop (ITW)}, pages 241--245. IEEE, 2016.

\bibitem{MonomialCartesian}
H.~H. L{\'o}pez, G.~L. Matthews, and I.~Soprunov.
\newblock Monomial-cartesian codes and their duals, with applications to lcd codes, quantum codes, and locally recoverable codes.
\newblock {\em Designs, Codes and Cryptography}, 88(8):1673--1685, 2020.

\bibitem{CartesianCode}
H.~H. L{\'o}pez, C.~Renter{\'\i}a-M{\'a}rquez, and R.~H. Villarreal.
\newblock Affine cartesian codes.
\newblock {\em Designs, Codes and Cryptography}, 71(1):5--19, 2014.

\bibitem{pardue}
K.~Pardue.
\newblock {\em Nonstandard Borel-fixed ideals}.
\newblock Brandeis University, 1994.

\bibitem{aed-pc}
C.~Pillet, V.~Bioglio, and I.~Land.
\newblock Polar codes for automorphism ensemble decoding.
\newblock In {\em 2021 IEEE Information Theory Workshop (ITW)}, pages 1--6. IEEE, 2021.

\bibitem{scinvariantpolar}
Z.~Ye, Y.~Li, H.~Zhang, R.~Li, J.~Wang, G.~Yan, and Z.~Ma.
\newblock The complete sc-invariant affine automorphisms of polar codes.
\newblock In {\em 2022 IEEE International Symposium on Information Theory (ISIT)}, pages 2368--2373. IEEE, 2022.

\end{thebibliography}

\end{document}